\documentclass%[english]
{amsart}
\usepackage[cmtip,all]{xy}
\usepackage{graphicx}
\usepackage{amssymb}

\usepackage{graphics}
\usepackage{epsf,epsfig,amsmath}
\usepackage{psfrag}

\newtheorem{theorem}{Theorem}[section]
\newtheorem{lemma}[theorem]{Lemma}
\newtheorem{prop}[theorem]{Proposition}

\theoremstyle{definition}
\newtheorem{definition}[theorem]{Definition}
\newtheorem{example}[theorem]{Example}

\theoremstyle{remark}

\numberwithin{equation}{section}

         \newcommand{\cO}{\mathcal O}

\newcommand{\ind}{{\rm Ind}}
\newcommand{\dom}{{\rm Dom}}

\begin{document}

\title{A simple uniform approach to complexes arising from forests}

\author{Mario Marietti}
\address{Universit\`a degli Studi di Roma ``La Sapienza'', Piazzale A. Moro 5, 00185 Roma, Italy}
\email{marietti@mat.uniroma1.it \hspace{10pt} www.mat.uniroma1.it/$\sim$marietti}

\author{Damiano Testa}
\address{Universit\`a degli Studi di Roma ``La Sapienza'', Piazzale A. Moro 5, 00185 Roma, Italy}
\email{testa@mat.uniroma1.it \hspace{23pt} www.mat.uniroma1.it/$\sim$testa}

\begin{abstract}
In this paper we present a unifying approach to study the homotopy type of 
several complexes arising from forests.  We show that this method applies 
uniformly to many complexes that have been extensively studied.
\end{abstract}

\maketitle

\section{Introduction}
In the recent years several complexes arising from forests have been studied 
by different authors with different techniques (see \cite{BM}, \cite{BLN}, 
\cite{EH}, \cite{E}, \cite{Kl}, \cite{K1}, \cite{K2}, \cite{MT1}, \cite{W}).  
The interest in these problems is motivated by applications in 
different contexts, such as graph theory and statistical mechanics (\cite{BK}, 
\cite{BLN}, \cite{J}).  
We introduce a unifying approach to study the homotopy type of 
several of these complexes.  With our technique we obtain simple 
proofs of results that were already known as well as new results.  
These complexes are wedges of spheres of (possibly) different 
dimensions and include, for instance, the complexes of directed trees, the 
independence complexes, the dominance complexes, the matching 
complexes, the interval order complexes.  In all cases our method provides a 
recursive procedure to compute the exact homotopy type of the simplicial complex.  
The dimensions of the spheres arising with these constructions are often 
strictly related to classical graph theoretical invariants of the underlying 
forest.  Thus we give a topological interpretation to these well-known 
combinatorial invariants.

Section 2 is devoted to notation and background.  In Section 3 we introduce the 
two basic concepts of this paper: the simplicial complex property of being a 
grape and the strictly related notion of domination between vertices of a 
simplicial complex.  In Section 4 we discuss several applications of these 
notions.

%We introduce a class of simplicial complexes that we call grapes, 
%satisfying are easily seen 
%to be homotopic to wedges of spheres.  
%present a unifying approach to study the homotopy type of 
%several complexes arising from forests.  We show that this method applies 
%uniformly to many complexes that have been extensively studied in the 
%recent years by different authors with different techniques (see \cite{BM}, 
%\cite{BLN}, \cite{EH}, \cite{E}, \cite{Kl}, \cite{K1}, \cite{K2}, \cite{MT1}, 
%\cite{W}).  Thus with our technique we obtain simple proofs of results that 
%were already known as well as new results.  
%The interest in these problems is motivated by applications in 
%different contexts, such as graph theory and statistical mechanics (\cite{BK}, 
%\cite{BLN}, \cite{J}).  
%These complexes are wedges of spheres of (possibly) different 
%dimensions and include, for instance, the complexes of directed trees, the 
%independence complexes, the dominance complexes, the matching 
%complexes, the interval order complexes.  In all cases our method provides a 
%recursive procedure to compute the exact homotopy type of the simplicial complex.

\section{Notation} \label{sedue}

%If $r \in \mathbb{Z}$, $r \geq 0$, we let $[r] := \{ 1 , \ldots , r\}$.  
%The cardinality of a set $A$ will be denoted by $|A|$.

Let $G=(V,E)$ be a graph (finite undirected graph with no loops or multiple edges).  
For all $S \subset V$, let 
$N[S] := \bigl\{ w \in V \; | \; \exists s \in S , \{ s , w \} \in E \bigr\} \cup S$ be the 
{\it closed neighborhood of $S$}; when $S = \{ v \}$, then we let $N[v] = N[\{v\}]$.  
If $S \subset V$, then $G \setminus S$ is the graph obtained by removing 
from $G$ the vertices in $S$ and all the edges having a vertex in $S$ as an 
endpoint.  Similarly, if $S \subset E$, then $G \setminus S$ is the graph obtained 
by removing from $G$ the edges in $S$.  A vertex $v \in V$ is a leaf if it belongs to 
exactly one edge.  
A set 
$D \subset V$ is called {\it dominating} if $N[D] = V$.  
A set $D \subset V$ is called {\it independent} if no two vertices in $S$ are adjacent, 
i.e.~$\{v,v'\}\notin E$ for all $v,v'\in D$.    
A {\it vertex cover of $G$} is a subset $C \subset V$ such that every edge of $G$ 
contains a vertex of $C$.  An {\it edge cover of $G$} is a subset $S \subset E$ such 
that the union of all the endpoints of the edges in $S$ is $V$.  
A {\it matching of $G$} is a subset $M \subset E$ of pairwise 
disjoint edges.  

We consider the following classical invariants of a graph $G$ which have been 
extensively studied by graph theorists (see, for instance, 
\cite{AL}, \cite{ALH}, \cite{BC}, \cite{ET}, \cite{HHS}, \cite{HY}); we let 
\begin{itemize}
\item $\gamma (G) := \min \bigl\{ |D|, D \text{ is a dominating set of $G$} \bigr\}$ be the 
{\it domination number of $G$}; 
\item $i (G) := \min \bigl\{ |D|, D \text{ is an independent dominating set of $G$} \bigr\}$ 
be the {\it independent domination number of $G$}; 
\item $\alpha _0 (G) := \min \bigl\{ |C|, C \text{ is a vertex cover of $G$} \bigr\}$ be the 
{\it vertex covering number of $G$}; 
%\item $\alpha _1 (G) := \min \bigl\{ |C|, C \text{ is an edge cover of $G$} \bigr\}$ be the 
%{\it edge covering number of $G$}; 
\item $\beta _1 (G) := \max \bigl\{ |M|, M \text{ is a matching of $G$} \bigr\}$ be the 
{\it matching number of $G$}. 
\end{itemize}

Recall the following well-known result of K\"onig (cf~\cite{D}, Theorem~2.1.1).
% and Gallai (cf~\cite{HHS}, Theorem~9.27).

\begin{theorem}[K\"onig] \label{konig}
Let $G$ be a bipartite graph.  Then $\alpha _0 (G) = \beta _1 (G)$.
\end{theorem}

%\begin{theorem}[Gallai] \label{gallai}
%Let $G = (V,E)$ be a graph without isolated vertices.  Then 
%%
%$$ \alpha _1 (G) + \beta _1 (G) = |V| . $$
%\end{theorem}

We refer the reader to~\cite{B} or~\cite{D} for all undefined notation on graph theory.

Let $X$ be a finite set of cardinality $n$.

\begin{definition}
A simplicial complex $\Delta $ on $X$ is a set of subsets of $X$, called {\it faces}, 
such that, if $\sigma \in \Delta $ and $\sigma ' \subset \sigma$, then $\sigma ' \in \Delta $.  
The faces of cardinality one are called {\it vertices}.
\end{definition}

We do not require that $x \in \Delta $ for all $x \in X$.

Every simplicial complex $\Delta $ on $X$ different from $\{ \emptyset \}$ has a 
standard geometric realization.  Let $W$ be the real vector space having $X$ as 
basis.  The realization of $\Delta $ is the union of the convex hulls of the sets 
$\sigma $, for each face $\sigma \in \Delta $.  Whenever we mention a topological 
property of $\Delta $, we implicitly refer to the geometric realization of $\Delta $.

As examples, we mention the $(n-1)-$dimensional simplex ($n \geq 1$) corresponding to the 
set of all subsets of $X$, its boundary (homeomorphic to the $(n-2)-$dimensional sphere) 
corresponding to all the subsets different 
from $X$, and the boundary of the $n-$dimensional cross-polytope, that is the 
dual of the $n-$dimensional cube.  Note that the cube, its boundary and the 
cross-polytope are not simplicial complexes.  
We note that the simplicial complexes $\{ \emptyset \}$ and 
$\emptyset $ are different: we call $\{ \emptyset \}$ the $(-1)-$dimensional 
sphere, and $\emptyset $ the $(-1)-$dimensional simplex, or the empty simplex.  
The empty simplex $\emptyset $ is contractible.
%For $n \geq 1$, let $S^{n-2} := R \bigl( (x_1 \cdots x_n , \icss ) \bigr)$, the {\it sphere of dimension $n-2$}.

Let $\sigma \subset X$ and define simplicial complexes 
$$ \begin{array}{ccl}
(\Delta : \sigma ) & := & \bigl\{ m \in \Delta \; | \; \; \sigma \cap m = \emptyset \,,\, 
m \cup \sigma \in \Delta \bigr\} \\[5pt]
(\Delta , \sigma ) & := & \bigl\{ m \in \Delta \; | \; \; \sigma \not \subset m \bigr\} . 
\end{array} $$
The simplicial complexes $(\Delta : \sigma )$ and $(\Delta , \sigma )$ are usually called 
link and face-deletion of $\sigma $.
If $\Delta _1 , \ldots , \Delta_k$ are simplicial complexes on $X$, we define 
$$ {\rm join} \bigl( \Delta_1 , \ldots , \Delta_k \bigr) := 
\bigl\{ \cup_{m_i \in \Delta _i} m_i \bigr\} .  $$
If $x,y \in X$, let 
$$ \begin{array}{rcl}
A_x \bigl( \Delta \bigr) & := & {\rm join} \bigl(\Delta, \{1,x\} \bigr) \\[5pt]
\Sigma_{x,y} \bigl( \Delta \bigr) & := & {\rm join} \bigl(\Delta, \{1,x,y\} \bigr) ; 
\end{array} $$
$A_x \bigl( \Delta \bigr)$ and $\Sigma_{x,y} \bigl( \Delta \bigr)$ 
are both simplicial complexes.
If $x \neq y$  and no face of $\Delta $ contains either of them, then 
$A_x \bigl( \Delta \bigr)$ and $\Sigma_{x,y} \bigl( \Delta \bigr)$ are called respectively 
the {\it cone on $\Delta$ with apex $x$} and the {\it suspension of $\Delta$}.  
If $x \neq y$ and $x' \neq y'$ are in $X$ and are not contained in any face of 
$\Delta $, then the suspensions 
$\Sigma_{x,y} \bigl( \Delta \bigr)$ and $\Sigma_{x',y'} \bigl( \Delta \bigr)$ are 
isomorphic; hence in this case sometimes we drop the subscript 
from the notation.
It is well-known that 
if $\Delta $ is contractible, then $\Sigma (\Delta )$ is contractible, and that 
if $\Delta $ is homotopic to a sphere of dimension $k$, then $\Sigma (\Delta )$ is 
homotopic to a sphere of dimension $k+1$.  
%We state the following easy lemma for further reference (see \cite{MT1} for a proof).
Note that for all 
%\begin{lemma} 
%Let $\Delta $ be a simplicial complex and let 
$x \in X$ we have 
\begin{equation} \label{precontra}
\Delta = A_x (\Delta : x) \cup _{(\Delta : x)} (\Delta , x) .
\end{equation}

We recall the notions of collapse and simple-homotopy (see \cite{C}).  
Let $\sigma \supset \tau $ be faces of a simplicial complex 
$\Delta $ and suppose that $\sigma $ is maximal and $|\tau | = | \sigma | -1$ (i.e.~$\tau $
has codimension one in $\sigma $).  
If $\sigma $ is the only face of $\Delta $ properly containing $\tau $, then the removal 
of $\sigma $ and $\tau $ is called an {\it elementary collapse}. If a simplicial complex 
$\Delta '$ is obtained from $\Delta$ by an elementary collapse,  we write 
$\Delta \succ \Delta'$.  
When $\Delta '$ is a subcomplex of $\Delta $, we say that {\it $\Delta $ collapses onto $\Delta '$} 
if there is a sequence of elementary collapses leading from $\Delta $ to $\Delta '$.

\begin{definition}
Two simplicial complexes $\Delta $ and $\Delta '$ are {\it simple-homotopic} if they are 
equivalent under the equivalence relation generated by $\succ $.
\end{definition}

It is clear that if $\Delta $ and $\Delta '$ are simple-homotopic, then they are also 
homotopic, and that a cone collapses onto a point.

\section{Domination and grapes} \label{setre}

In this section we introduce the notions of grape and domination between vertices 
of a simplicial complex $\Delta $, and we give some consequences on the topology 
of $\Delta $.

\begin{definition} \label{grappolo}
A simplicial complex $\Delta $ is a {\it grape} if 
\begin{enumerate}
\item there is $a \in X$ such that $(\Delta : a)$ is contractible in $(\Delta , a)$ and 
both $(\Delta , a)$ and  $(\Delta : a)$ are grapes, or 
\item $\Delta $ has at most one vertex.
\end{enumerate}
\end{definition}

Note that if $\Delta $ is a cone with apex $b$, then $\Delta $ is a grape; indeed for any 
vertex $a \neq b$ we have that both $(\Delta , a)$ and $(\Delta : a)$ are cones with apex 
$b$, thus $(\Delta : a)$ is contractible in $(\Delta , a)$ and we conclude by induction.

%\begin{lemma}\label{pabba}
%If $a \in X$ and $(\Delta : a)$ is contractible in $(\Delta , a)$, then 
%$\Delta \simeq (\Delta , a) \vee \Sigma (\Delta : a)$.
%\end{lemma}
%
%\begin{proof}
%Follows at once from \cite[Proposition~0.18]{Ha}.
%\end{proof}

\begin{prop}
If $\Delta $ is a grape, then $\Delta $ is contractible or homotopic to a wedge of spheres.
\end{prop}

\begin{proof}
Proceed by induction on the number $n$ of vertices of $\Delta $.  If $n\leq 1$, then 
the result is clear.  If $n \geq 2$, by definition of a grape, there is a vertex $a$ 
such that $(\Delta : a)$ is contractible in $(\Delta , a)$.  By equation (\ref{precontra}) 
and~\cite[Proposition~0.18]{Ha} 
we deduce that $\Delta \simeq (\Delta , a) \vee \Sigma (\Delta : a)$.  Thus 
the result follows by induction on the number of vertices of $\Delta $ 
from the definition of grape.
%, equation (\ref{precontra}) and Lemma~\ref{pabba}.
\end{proof}

In fact we proved that if $a \in X$ and $(\Delta : a)$ is contractible in 
$(\Delta , a)$, then $\Delta \simeq (\Delta , a) \vee \Sigma (\Delta : a)$.  
As a consequence, if $\Delta $ is a grape, keeping track of the elements $a$ 
of Definition~\ref{grappolo}, we have a recursive procedure to compute the 
number of spheres of each dimension of the wedge.

In order to prove that a simplicial complex $\Delta $ is a grape we need to 
find a vertex $a$ such that $(\Delta : a)$ is contractible in $(\Delta , a)$; 
in the applications we will prove the stronger statement that there is a cone 
$C$ such that $(\Delta : a) \subset C \subset (\Delta , a)$ (or equivalently 
if $A_b \bigl(\Delta : a \bigr) \subset (\Delta , a)$).  In the two extreme 
cases $C = (\Delta , a)$ or $C = (\Delta : a)$, we have 
$\Delta \simeq \Sigma (\Delta : a)$ or $\Delta \simeq (\Delta , a)$ respectively 
(in the latter case $\Delta $ collapses onto $(\Delta , a)$).

\begin{definition}
Let $a,b \in X$; {\it $a$ dominates $b$ in $\Delta $} if there is a cone $C$ with 
apex $b$ such that $(\Delta : a) \subset C \subset (\Delta , a)$.
\end{definition}

In the special case in which $C=(\Delta , a)$ we obtain~\cite[Definition~3.4]{MT1}.

\section{Applications}

In this section we use the concepts introduced in Section~\ref{setre} 
to study simplicial complexes associated to forests.  We shall see 
that all these complexes are grapes (and hence they are homotopic to 
wedges of spheres) by giving in each case the graph 
theoretical property corresponding to domination.

\subsection{Oriented forests}

We study the simplicial complex of oriented forests of a multidigraph.  
In the case of directed graphs, this concept coincides with the one 
introduced in~\cite{K1} by D.~Kozlov (following a suggestion of 
R.~Stanley) who called it the complex of directed trees.  The reason 
to generalize this notion to multidigraphs is to allow an inductive 
procedure to work.

A multidigraph $G$ is a pair $(V,E)$, where $V$ is a finite set of elements called 
vertices and $E \subset V \times V \times \mathbb{N}$ is a finite set of elements 
called edges.  If $(x,y,n) \in E$ 
we write $x \to_n y$, or simply $x \to y$ when no confusion is possible, and call 
it an edge with source $x$ and target $y$, or more simply an edge from $x$ to $y$.  
We usually identify $G = (V,E)$ with $G'=(V',E')$ if there are two bijections 
$\varphi : V \to V'$ and $\psi : E \to E'$ such that 
$\psi \bigl( x,y,n \bigr) = \bigl( \varphi(x) , \varphi (y) , n' \bigr)$, 
for some $n' \in \mathbb{N}$.  
A multidigraph $H = (V',E')$ is a subgraph of $G$ if $V' \subset V$ and $E' \subset E$.  
A directed graph is a multidigraph such that distinct edges cannot have both same 
source and same target.
%
%A multidigraph $G$ is a pair $(V,E)$, where $V$ is a finite set of elements called 
%vertices and $E$ is a multiset of ordered pairs of vertices called edges.  If $(x,y) \in E$ 
%we write $x \to y$ and call it an edge with source $x$ and target $y$, or more simply 
%an edge from $x$ to $y$.  When the edge $e=x \to y$ has multiplicity $n$, we think of 
%$e$ as $n$ distinct edges with same source $x$ and same target $y$.  When we write 
%$x \to y$ we mean that we have chosen a specific edge joining $x$ to $y$.  A 
%multidigraph $H = (V',E')$ is a subgraph of $G$ if $V' \subset V$ and $E' \subset E$.
%
%Let $G=(V,E)$ be a loopless directed graph, that is $V$ is a finite set and 
%$E$ is any subset of $V \times V \setminus \bigl\{ (v,v) \, | \, v \in V \bigr\}$.  
%If $(x,y) \in E$ we write $x \to y$ and call it an edge with source $x$ and 
%target $y$, or more simply an edge from $x$ to $y$.  A directed graph $H = (V',E')$ 
%is a subgraph of $G$ if $V' \subset V$ and $E' \subset E$.  
%
We associate to a multidigraph $G=(V,E)$ its underlying undirected graph 
$G^u$ with vertex set $V$ and where $x,y$ are joined by an edge in $G^u$ if and only if 
$x \to y$ or $y \to x$ are in $E$.

An oriented cycle of $G$ is a connected subgraph $C$ of $G$ such that each vertex of 
$C$ is the source of exactly one edge and target of exactly one edge.  
An oriented forest is a multidigraph $F$ such that $F$ contains no oriented cycles 
and different edges have distinct targets.

\begin{definition}
The complex of oriented forests of a multidigraph $G=(V,E)$ is the simplicial complex 
$OF(G)$ whose faces are the subsets of $E$ forming oriented forests.
\end{definition}

If $e$ is a loop, i.e.~an edge of $G$ with source equal to its target, then 
$OF(G) = OF \bigl( G \setminus \{e\} \bigr)$.  Thus, from now on, we ignore 
the loops.

It follows from the definitions that the complex $OF(G)$ is a cone with apex 
$y \to x$ if and only if $y \to x$ is the unique edge with target $x$ and there 
are no oriented cycles in $G$ containing $y \to x$.  
For any $z \to _n u \in E$, the simplicial complex $(OF(G),z \to _n u)$ is the 
complex of oriented forests of the multidigraph 
$\bigl( V , E \setminus \{ z \to _n u \} \bigr)$.  

We denote by $G _{\downarrow z \to u}$ the multidigraph 
%whose vertex set is 
%$V \setminus \{z\}$ and whose edges are of the following three types: 
%%
%\begin{itemize}
%\item $a \to b \in E$, with $b \neq u$, $a,b \neq z$; 
%\item $a \to u$, with $a \to z \in E$ and $a \neq u$; 
%\item $u \to b$, with $z \to b \in E$ and $b \neq u$. 
%%\item $z \to c$, with $u \to c \in E$ and $c \neq z$.
%\end{itemize}
%%
%Roughly speaking, the graph $G _{\downarrow z \to u}$ is 
obtained from $G$ by first removing the edges with target $u$, and then 
identifying the vertex $z$ with the vertex $u$.  
The reason for introducing this multidigraph is that $\bigl( OF(G) : z \to u \bigr) = 
OF \bigl( G _{\downarrow z \to u}\bigr)$.  Indeed 
no face of $\bigl( OF(G) : z \to u \bigr)$ contains an arrow with target 
$u$ or becomes an oriented cycle by adding $z \to u$; thus there is a correspondence 
between the faces of the two complexes.  We note that if $G$ is a directed graph, 
then $G _{\downarrow z \to u}$ could be a multidigraph which is not a directed graph.
\begin{center}
\begin{minipage}{150pt}
$$ \xymatrix { z \ar[rr] \ar[dr] & & u \ar[dl] \\
& x} $$
\medskip
\centerline{A directed graph $G$}
\end{minipage} 
\begin{minipage}{150pt}
$$ \xymatrix { u \ar@<-.5ex>[d]  \ar@<.5ex>[d]\\
x } $$
\medskip
\centerline{The multidigraph $G_{\downarrow z \to u}$}
\end{minipage} 
\end{center}

\begin{lemma} \label{dfdomi}
Let $z \to u$ and $y \to x$ be distinct vertices of $OF(G)$; then $z \to u$ dominates 
$y \to x$ in $OF(G)$ if and only if one of the following is satisfied:
\begin{itemize}
\item $z=y$ and $u=x$; 
\item $u=x$ and there are no oriented cycles containing $y \to x$; 
\item $u \neq x$, $y \to x$ is the unique edge with target $x$, and all oriented cycles 
containing $y \to x$ contain also $u$.
\end{itemize}
\end{lemma}

\begin{proof}
It is clear that $z \to _n u$ dominates $z \to _m u$ whenever $m \neq n$.  Thus we assume 
that $(z,u) \neq (y,x)$.

Let $z \to u$ dominate $y \to x$ in $OF(G)$.  Suppose that $u=x$.  By contradiction, let $C$ be 
an oriented cycle of $G$ containing $y \to x$.  Then $z \to u \notin C$ and hence the edges of 
$C \setminus \{y \to x\}$ are a face of $\bigl( OF(G) : z \to u \bigr)$, but the edges of $C$ are 
not a face of $\bigl( OF(G) , z \to u \bigr)$ and hence $\bigl( OF(G) , z \to u \bigr)$ does not 
contain a cone with apex $y \to x$.  Suppose now that $u \neq x$.  Clearly there can be no edges 
with target $x$ different from $y \to x$, since each of these edges forms a face of 
$\bigl( OF(G) : z \to u \bigr)$.  Let $C$ be an oriented cycle of $G$ containing $y \to x$.  
Then the edges of 
$C \setminus \{y \to x\}$ are a face of $\bigl( OF(G) : z \to u \bigr)$ if and only if $C$ does not 
contain the vertex $u$.  Since the edges of $C$ are 
not a face of $\bigl( OF(G) , z \to u \bigr)$ we must have that $u$ is a vertex of $C$.

Conversely, let $\sigma $ be a face of $\bigl( OF(G) : z \to u \bigr)$.  
We need to show that $\sigma \cup \{y \to x\}$ is a face of $\bigl( OF(G) , z \to u \bigr)$: 
equivalently we need to show that it is a face of $OF(G)$, since $\sigma $ does not contain 
$z \to u$.  We may assume that 
$y \to x \notin \sigma $.  Suppose first that $u=x$ and there are no oriented cycles 
containing $y \to x$; $\sigma $ contains no edge with target $x$, since 
$\sigma \in \bigl( OF(G) : z \to u \bigr)$ and $\sigma \cup \{y \to x\}$ is a face of $OF(G)$ 
since there are no oriented cycles containing $y \to x$.  Suppose now that 
$u \neq x$, $y \to x$ is the unique edge with target $x$, and all oriented cycles 
containing $y \to x$ contain also $u$.  By assumption no edge of $\sigma $ has $x$ as a 
target; moreover if $C$ is a cycle containing $y \to x$, then $\sigma $ cannot contain all the 
edges of $C \setminus \{y \to x\}$, since one of these edges has target $u$ and so it is 
not a face of $\bigl( OF(G) : z \to u \bigr)$.
\end{proof}

We call a multidigraph $F$ a multidiforest if its underlying graph $F^u$ is a forest.  
The following result determines the homotopy types of the complexes of oriented forests 
of multidiforests.

\begin{theorem} \label{mudifo}
Let $F$ be a multidiforest.  Then $OF(F)$ is a grape.
\end{theorem}

\begin{proof}
Proceed by induction on the number of edges of $F$.  It suffices to show that $F$ contains 
two distinct edges $z \to u$ and $y \to x$ such that $z \to u$ dominates $y \to x$, since 
both $F \setminus \{z \to u\}$ and $F _{\downarrow z \to u}$ are multidiforests.

If $x \to_n y , x \to _m y$ are distinct edges, then $x \to _n y$ dominates $x \to _m y$ 
(and conversely) by Lemma~\ref{dfdomi}.  Thus we may assume that $F$ is a directed graph.  
Let $y$ be a leaf of $F^u$ and let $x$ be the vertex adjacent to $y$.  
Recall that the complex $OF(G)$ is a cone with apex 
$a \to b$ if and only if $a \to b$ is the unique edge with target $b$ and there 
are no oriented cycles in $F$ containing $a \to b$ (i.e.~there is no edge with source 
$b$ and target $a$).  Since a cone is a grape, we only need to consider two cases: 
\begin{enumerate}
\item $y \to x$ and $x \to y$ are both edges of $F$, \label{doppi}
\item $y \to x$ is an edge of $F$, $x \to y$ is not and there is $z \to x$ with $z \neq y$. \label{singoli}
\end{enumerate}
By Lemma~\ref{dfdomi}, in case (\ref{doppi}) $y \to x$ dominates $x \to y$, 
in case (\ref{singoli}) $z \to x$ dominates $y \to x$; in both cases we conclude.
%In both cases we conclude by induction on the number of edges of $F$ since 
%$\bigl( OF(F) : y \to x \bigr)$ and $\bigl( OF(F) : z \to x \bigr)$ are complexes of 
%oriented forests of forests with fewer edges than $F$.
\end{proof}

%, giving a recursive procedure to compute explicitly the 
%homotopy type of $OF(F)$, i.e.~the number of spheres of each dimension.  

The proof of Theorem~\ref{mudifo} gives a recursive procedure to compute explicitly the 
homotopy type of $OF(F)$, i.e.~the number of spheres of each dimension.  
Thus it generalizes~\cite[Section~4]{K1}, where a recursive procedure to compute the homology 
groups of the complexes of oriented forests of directed trees is given.

\begin{example}
Let $F$ be the directed tree depicted in the following figure.
\begin{center}
\begin{minipage}{150pt}
$$ \xymatrix{ a \ar[dr] &&&& f \ar[dl] \\
& c \ar@<.5ex>[r] & d \ar@<.5ex>[l] \ar@<.5ex>[r] & e \ar@<.5ex>[l] \\
b \ar[ur] &&&& g \ar[ul] } $$

\medskip
\centerline{The directed tree $F$}
\end{minipage} 
\end{center}
By Lemma~\ref{dfdomi}, $d \to c$ dominates $a \to c$ and hence $OF(F) \simeq 
OF (F_1) \vee \Sigma OF (F_2)$, where the directed trees $F_1,F_2$ are given 
in the following figure.
\begin{center}
\begin{minipage}{150pt}
$$ \xymatrix{ a \ar[dr] &&&& f \ar[dl] \\
& c \ar@<.5ex>[r] & d \ar@<.5ex>[r] & e \ar@<.5ex>[l] \\
b \ar[ur] &&&& g \ar[ul] } $$

\medskip
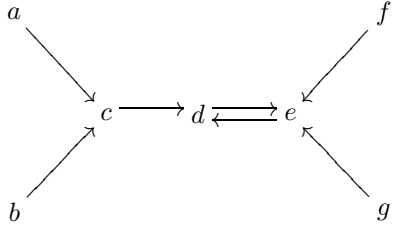
\centerline{The directed tree $F_1$}
\end{minipage} 
\hspace{10pt}
\begin{minipage}{150pt}
$$ \xymatrix { && f \ar[dl] \\
d \ar@<.5ex>[r] & e \ar@<.5ex>[l] \\
&& g \ar[ul] } $$

\medskip
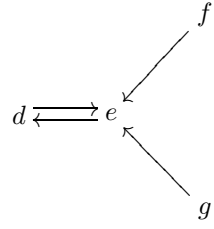
\centerline{The directed tree $F_2$}
\end{minipage} 
\end{center}
We consider first $OF(F_2)$.  The edge $d \to e$ dominates $f \to e$ in $OF(F_2)$; 
the complex $\bigl( OF(F_2) , d \to e \bigr)$ is a cone with apex $e \to d$, and 
$\bigl( OF(F_2) : d \to e \bigr) = \{ \emptyset \}$, since $F_2 {}_{\downarrow d \to e}$ 
has no edges different from loops.  Hence $OF(F_2) \simeq S^0$ (and it is as depicted below) 
and $OF(F) \simeq OF(F_1) \vee S^1$.
\begin{center}
\begin{minipage}{150pt}
$$ \xygraph {[] !~:{@{.}} 
!{<0pt,0pt>;<20pt,0pt>:} 
%[d] {\bullet} 
%!{\save +<-13pt,0pt>*\txt{$\scriptstyle \{c,d\}$}  \restore}
%[urr] 
{\bullet} 
!{\save +<-13pt,0pt>*\txt{$\scriptstyle f \to e$}  \restore}
[rr] {\bullet} 
!{\save +<13pt,0pt>*\txt{$\scriptstyle e \to d$}  \restore}
[dd] {\bullet} 
!{\save +<13pt,0pt>*\txt{$\scriptstyle g \to e$}  \restore}
[ll] {\bullet} 
!{\save +<-13pt,0pt>*\txt{$\scriptstyle d \to e$}  \restore}
[uu] - [rr] - [dd] } $$

\medskip
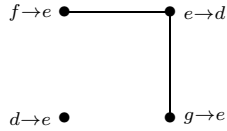
\centerline{The simplicial complex $OF(F_2)$}
\end{minipage} 
\end{center}

Let us now consider $OF(F_1)$.  
By Lemma~\ref{dfdomi}, $a \to c$ dominates $b \to c$.  Since 
$\bigl( OF(F_1) , a \to c \bigr)$ is a cone with apex $b \to c$, it follows 
that $OF(F_1) \simeq \Sigma OF(F_3)$, where $F_3$ is depicted 
in the following figure.
\begin{center}
\begin{minipage}{150pt}
$$ \xymatrix{ &&& f \ar[dl] \\
c \ar@<.5ex>[r] & d \ar@<.5ex>[r] & e \ar@<.5ex>[l] \\
&&& g \ar[ul] } $$

\medskip
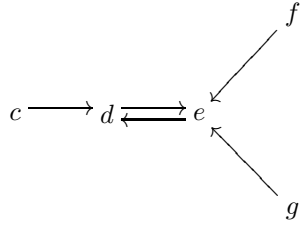
\centerline{The directed tree $F_3$}
\end{minipage} 
\end{center}
The edge $e \to d$ dominates $c \to d$ in $OF(F_3)$; 
$\bigl( OF(F_3) , e \to d \bigr)$ is a cone with apex $c \to d$, and 
$\bigl( OF(F_3) : e \to d \bigr)$ consists of the two isolated points 
$f \to e$ and $g \to e$.  
Thus $OF(F_3) \simeq S^1$; indeed $OF(F_3)$ is depicted in the following 
figure.
\begin{center}
\begin{minipage}{150pt}
$$ \xygraph {[] !~:{@{.}} 
!{<0pt,0pt>;<20pt,0pt>:} 
%[d] {\bullet} 
%!{\save +<-13pt,0pt>*\txt{$\scriptstyle \{c,d\}$}  \restore}
%[urr] 
{\bullet} 
!{\save +<-13pt,0pt>*\txt{$\scriptstyle f \to e$}  \restore}
[rr] {\bullet} 
!{\save +<13pt,0pt>*\txt{$\scriptstyle e \to d$}  \restore}
[dd] {\bullet} 
!{\save +<13pt,0pt>*\txt{$\scriptstyle g \to e$}  \restore}
[ll] {\bullet} 
!{\save +<-13pt,0pt>*\txt{$\scriptstyle c \to d$}  \restore}
[ld]{\bullet} 
!{\save +<-13pt,0pt>*\txt{$\scriptstyle d \to e$}  \restore}
[lr]-[ur]-[uu] - [rr] - [dd]-[ll] }
$$

\medskip
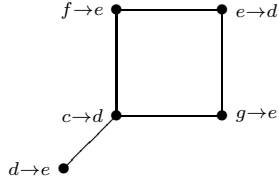
\centerline{The simplicial complex $OF(F_3)$}
\end{minipage} 
\end{center}
%\begin{minipage}{150pt}
%$$ \xygraph {[] !~:{@{.}} 
%!{<0pt,0pt>;<20pt,0pt>:} 
%{\bullet} 
%!{\save +<-13pt,0pt>*\txt{$\scriptstyle f \to e$}  \restore}
%[rr] {\bullet} 
%!{\save +<13pt,0pt>*\txt{$\scriptstyle e \to d$}  \restore}
%[dd] {\bullet} 
%!{\save +<13pt,0pt>*\txt{$\scriptstyle g \to e$}  \restore}
%[ll] {\bullet} 
%!{\save +<-13pt,0pt>*\txt{$\scriptstyle d \to e$}  \restore}
%[ur] {\bullet} 
%!{\save +<-13pt,0pt>*\txt{$\scriptstyle c \to d$}  \restore}
%[rl] - [ul] - [rr] - [dd] - [ul] - [dl] } $$
%\medskip
%\centerline{The simplicial complex $OF(F_3)$}
%\end{minipage} 
%
%
Finally the simplicial complex $OF(F)$ is homotopic to $S^2 \vee S^1$.
\end{example}

\subsection{The independence complex} \label{secin}

Let $G = (V,E)$ be a graph.  The simplicial complex on $V$ whose faces are the 
subsets of $V$ containing no adjacent vertices is denoted by $\ind (G)$ and is 
called the {\it independence complex of $G$}.  We have 
\begin{equation} \label{proide}
\begin{array}{rcl}
\bigl( \ind (G) , v \bigr) & = & \ind \bigl( G \setminus \{v\} \bigr) \\[5pt]
\bigl( \ind (G) : v \bigr) & = & \ind \bigl( G \setminus N[v] \bigr) . 
\end{array} 
\end{equation}
The simplicial complex $\ind (G)$ is a cone of apex $a$ if and only 
if $a$ is an isolated vertex of $G$.

\begin{lemma} \label{indo}
Let $a$ and $b$ be vertices of $G$; $a$ dominates $b$ in $\ind (G)$ 
if and only if $N[b] \setminus \{b\} \subset N[a]$.
\end{lemma}

\begin{proof}
The faces of $\ind \bigl( G \setminus N[a] \bigr)$ are the independent sets of vertices of 
$G \setminus N[a]$.  Let $D$ be a face of $\ind \bigl( G \setminus N[a] \bigr)$; 
$D \cup \{b\}$ is a face of $\ind \bigl( G \setminus a \bigr)$ if and only if $b \in D$ or 
$b \notin N[D]$.  Since this must be true for all faces, 
$N[b] \setminus \{b\} \cap \bigl(V \setminus N[a] \bigr) = \emptyset $, and the result follows.
\end{proof}

\begin{lemma} \label{scremo}
Let $a$ be a vertex of $G$ having distance two from a leaf $b$.  Then $\ind \bigl( G \bigr)$ 
collapses onto $\ind \bigl( G \setminus a \bigr)$.
\end{lemma}

\begin{proof}
Since $N[b] \setminus \{b\} \subset N[a]$, $a$ dominates $b$ by Lemma~\ref{indo}.  Moreover 
the simplicial complex $\ind \bigl( G \setminus N[a] \bigr)$ is a cone with apex $b$, since 
$G \setminus N[a]$ contains $b$ as an isolated vertex.  If 
$(\sigma _1 \supset \tau _1)$, \ldots , $(\sigma _r \supset \tau _r)$ is a sequence of elementary collapses 
of $\ind \bigl( G \setminus N[a] \bigr)$ onto $\emptyset $, then 
$\bigl( \sigma _1 \cup \{a\} \supset \tau _1 \cup \{a\} \bigr)$, \ldots , 
$\bigl( \sigma _r \cup \{a\} \supset \tau _r \cup \{a\} \bigr)$ is a sequence of elementary 
collapses of $\ind \bigl( G \bigr)$ onto $\ind \bigl( G \setminus a \bigr)$.
\end{proof}

The removal of vertices at distance two from a leaf has already been used by Kozlov 
for the independence complex of a path and by Wassmer for rooted forests 
(see~\cite{K1} and~\cite{W}).

In a forest $F$, a vertex $a$ dominates a vertex $b$ if and only if 
\begin{enumerate}
\item $b$ is a leaf and $a$ is adjacent to $b$; 
\item $b$ is a leaf and $a$ has distance two from $b$; 
\item $b$ is isolated. 
\end{enumerate}
The third case deals with the trivial case in which $\ind (F)$ is a cone with apex $b$.  
Specifying the treatment of the domination to the first case we obtain the 
analysis of~\cite[Section 6]{MT1}; specifying it to the 
second case we obtain the analysis of~\cite[Section 3.2]{W}.  In the first approach what 
happens is that at each stage the removal of the vertex $a$ and of all its neighbors changes 
the homotopy type of $\ind (F)$ by a suspension; thus the relevant informations are the 
number $r_1$ of steps required to reach a graph $F_1$ with no edges and the number $i_1$ 
of isolated vertices of $F_1$.  In the second approach what 
happens is that at each stage the removal of the vertex $a$ does not change the homotopy 
type of $\ind (F)$; thus the relevant informations are the numbers $r_2$ and $i_2$ of isolated 
edges and vertices of the graph $F_2$ obtained by performing the removal as long as possible.  
The conclusion is that $i_1 \neq 0$ if and only if $i_2 \neq 0$ if and only if $\ind (F)$ collapses 
onto a point.  If $i_1 = i_2 = 0$, then $r_1 =r_2=r$ and $\ind (F)$ collapses onto the 
boundary of the $r-$dimensional cross-polytope; it can be proved that $r = i(F) = \gamma (F)$, 
see~\cite[Theorem~6.4]{MT1}.  
We state explicitly the following result for further reference.

\begin{theorem} \label{inte}
Let $F$ be a forest.  Then $\ind (F)$ is a grape.  Moreover, $\ind (F)$ is either 
contractible or homotopic to a sphere.
\end{theorem}

\subsection{The dominance complex} \label{seset}

Let $G=(V,E)$ be graph.  The simplicial complex on $V$ whose faces are the complements of the 
dominating sets is denoted by $\dom (G)$ and is called the {\it dominance complex of $G$}; 
equivalently the minimal non-faces of $\dom (G)$ are the minimal elements of 
$\bigl\{ N[x] \,|\, x \in V \bigr\}$ .  The dominance complex of $G$ is never a cone.  
Let $a \in V$; we have 
$$ \bigl( \dom (G) : a \bigr) = \left( \dom \bigl( G \setminus a \bigr) , N[a] \setminus \{a\}  \right) . $$

\begin{lemma} \label{dodo}
Let $a, b$ be distinct non-isolated vertices of $G$; 
$a$ dominates $b$ in $\dom (G)$ if and only if 
for all $v \in N[b] \setminus N[a]$ there exists $m \in V$ such that 
$N[m] \setminus \{a\} \subset N[v] \setminus \{b\}$.
\end{lemma}

\begin{proof}
($\Rightarrow$)  Let $v \in N[b] \setminus N[a]$ and consider $\sigma := N[v] \setminus \{b\}$.  Since 
$\sigma \cup \{b\} \notin \Delta $ and $a$ dominates $b$, it follows that $\sigma \notin (\Delta : a)$.  
Thus there is $m \in V$ such that $N[m] \setminus \{a\} \subset \sigma = N[v] \setminus \{b\}$.

\noindent
($\Leftarrow$)  Proceed by contradiction and suppose that $a$ does not dominate $b$; hence there exists 
$\sigma \in (\Delta : a)$ such that $\sigma \cup \{b\} \notin \Delta $.  This means that 
\begin{enumerate}
\item $\nexists \; m \in V$ such that $N[m] \subset \sigma \cup \{a\}$, \label{none}
\item $\exists \: v \in V$ such that $N[v] \subset \sigma \cup \{b\}$. \label{sine}
\end{enumerate}
If $v$ satisfies $N[v] \subset \sigma \cup \{b\}$, then $N[v] \not \subset \sigma $, since otherwise 
(\ref{none}) would not hold.  
Thus $b \in N[v]$; moreover $a \notin N[v]$, since $N[v] \subset \sigma \cup \{b\}$ and $a \notin \sigma $.  
Hence $v \in N[b] \setminus N[a]$.  By assumption there is $m \in V$ such that $N[m] \setminus \{a\} \subset N[v] \setminus \{b\}$ 
and hence $N[m] \subset N[v] \cup \{a\} \setminus \{b\} \subset \sigma \cup \{a\}$, contradicting (\ref{none}).
\end{proof}

\begin{lemma} \label{doscremo}
Let $a,b,c \in V$ and suppose that $N[b]=\{a,b\}$ and $\{a,b,c\} \subset N[a]$.  
Then $\dom (G)$ collapses onto 
$\dom \bigl( G \setminus {\text{edge }} \{a,c\} \bigr)$.
\end{lemma}

\begin{proof}
Thanks to Lemma~\ref{dodo} $a$ dominates $b$ since $N[a] \supset N[b]$; 
$\bigl( \dom (G) , a \bigr)$ is a cone with apex $b$.  
Let $L = \bigl( \dom (G) : N[c] \setminus \{a\} \bigr) \subset \bigl( \dom (G) , a \bigr)$.  
%\bigl\{ \sigma \;|\; N[c] \setminus \{a\} \cup \sigma \subset \dom (G) \bigr\} \subset \bigl( \dom (G) , a \bigr)$.  
The simplicial complex $L$ is a cone with apex $b$.  
Let $(\sigma _1 \supset \tau _1)$, \ldots , $(\sigma _r \supset \tau _r)$ be a sequence of 
elementary collapses of $L$ to $\emptyset $; adding to $\sigma _i$ and $\tau _i$ the 
face $N[c] \setminus \{a\}$ for $1 \leq i \leq r$, we obtain a sequence of elementary collapses 
of $\dom (G)$ onto the simplicial complex $\bigl( \dom (G) , N[c] \setminus \{a\} \bigr)$.  It 
remains to show that $\bigl( \dom (G) , N[c] \setminus \{a\} \bigr) = 
\dom \bigl( G \setminus {\text{edge }} \{a,c\} \bigr)$.  The minimal non-faces of 
$\bigl( \dom (G) , N[c] \setminus \{a\} \bigr)$ and 
$\dom \bigl( G \setminus {\text{edge }} \{a,c\} \bigr)$ are respectively the minimal elements of 
$$ \bigl\{ N[v] \,|\, v \in V \bigr\} \cup \bigl\{ N[c] \setminus \{a\} \bigr\} $$
and the minimal elements of 
$$ \bigl\{ N[v] \,|\, v \in V \setminus \{a,c\} \bigr\} \cup \bigl\{ N[c] \setminus \{a\} , 
N[a] \setminus \{c\} \bigr\} , $$
where by $N[v]$ we mean the closed neighborhood of $v$ in the graph $G$.  Since 
$N[b] \subset N[a] \setminus \{c\}$, the minimal elements of the two sets above are 
the same.
\end{proof}

We now consider the dominance complex of a forest $F$.  Iterating as long as we can 
the removal of an edge satisfying the conditions of Lemma~\ref{doscremo}, we obtain 
a subforest $F'$ of $F$ containing only isolated vertices and edges.  The forest $F'$ 
depends on the choices of edges; the number $r$ of edges of $F'$, though, is independent 
of the choices thanks to the following result.

\begin{prop} \label{dosemplice}
Let $F$ be a forest. Then 
\begin{enumerate}
\item $\dom (F)$ is a grape; \label{doseuno}
\item $\dom (F)$ collapses onto the boundary of an $r-$dimensional cross-polytope, 
where $r$ is the number of edges of $F'$. \label{dosedue}
\end{enumerate}
\end{prop}

\begin{proof}
(\ref{doseuno}) By Lemma~\ref{dodo} the vertex $a$ adjacent to a leaf $b$ dominates $b$, since 
$N[a] \supset N[b]$.  The complex $(\dom(F) , a )$ is a cone with apex $b$, and 
$(\dom (F) : a) = \dom \bigl( F \setminus a \bigr)$.  Hence 
the result follows by induction on the number of vertices.

\noindent
(\ref{dosedue})  It follows at once from Lemma~\ref{doscremo} that $\dom (F)$ 
collapses onto $\dom (F')$.  Since the dominance complex of $F'$ is the boundary 
of the cross-polytope of dimension $r$, where $r$ is the number of edges of $F'$, 
the result follows.
\end{proof}

It can be proved that $r = \alpha _0 (F) = \beta _1 (F)$ (see~\cite[Theorem~8.1]{MT1}).

\subsection{Matching complex}

Let $G=(V,E)$ be a graph.  We define a simplicial complex $M(G)$ on $E$ 
whose faces are the matchings of $G$, i.e.~sets of pairwise disjoint edges.  We 
note that $M(G)$ is the independence complex of the line dual of $G$, i.e.~of the 
graph whose vertices are the edges of $G$ and where $\{e_1,e_2\}$ is an edge 
if $e_1 \neq e_2$ and $e_1 \cap e_2 \neq \emptyset $.  Note that if $e = \{x,y\} \in E$, 
then $\bigl( M(G) , e \bigr) = M(G \setminus e)$ and 
$\bigl( M(G) : e \bigr) = M \bigl( G \setminus \{x\} \setminus \{y\} \bigr)$.  
%Unless $G$ is a disjoint union of paths, the line dual of $G$ is not a forest.

If $F$ is a forest, then the line dual of $F$ is not a forest unless $F$ is a 
disjoint union of paths.  Hence Theorem~\ref{inte} does not apply to $M(F)$.  
Nevertheless, we have the following result.

\begin{theorem}
Let $F = (V,E)$ be a forest.  Then $M(F)$ is a grape.
\end{theorem}

\begin{proof}
We proceed by induction on the number of edges of $F$, the base case being obvious.  
Let $b$ be a leaf and let $a$ be adjacent to $b$.  If the edge $\{a,b\}$ is isolated, 
then $M(F)$ is a cone with apex $\{a,b\}$ and hence it is a grape.  Otherwise let $c \neq b$ 
be adjacent to $a$.  By Lemma~\ref{indo}, the edge $\{a,c\}$ dominates the edge 
$\{a,b\}$ in $M(F)$.  The result follows by induction since $\bigl( M(F) , \{a,c\} \bigr)$ 
and $\bigl( M(F) : \{a,c\} \bigr)$ are matching complexes of forests.
\end{proof}

\begin{example}
The simplicial complex $M(F)$ may be a wedge of spheres of different dimensions.  
Let $F$ be the tree depicted in the following figure.
\begin{center}
\begin{minipage}{150pt}
$$ \xygraph {[] !~:{@{.}} 
!{<0pt,0pt>;<20pt,0pt>:} 
{\bullet} 
!{\save +<-8pt,0pt>*\txt{$\scriptstyle a$}  \restore}
[dd] {\bullet} 
!{\save +<-8pt,0pt>*\txt{$\scriptstyle b$}  \restore}
[ur] {\bullet} 
!{\save +<0pt,7pt>*\txt{$\scriptstyle c$}  \restore}
[r] {\bullet} 
!{\save +<0pt,8pt>*\txt{$\scriptstyle d$}  \restore}
[ur] {\bullet} 
!{\save +<8pt,0pt>*\txt{$\scriptstyle e$}  \restore}
[dd] {\bullet} 
!{\save +<8pt,0pt>*\txt{$\scriptstyle f$}  \restore}
[rl] - [ul] - [ur] [dl] - [l] - [ul] [dd] - [ur] } $$

\medskip
\centerline{The tree $F$}
\end{minipage} 
\hspace{10pt}
\begin{minipage}{150pt}
$$ \xygraph {[] !~:{@{.}} 
!{<0pt,0pt>;<20pt,0pt>:} 
[d] {\bullet} 
!{\save +<-13pt,0pt>*\txt{$\scriptstyle \{c,d\}$}  \restore}
[urr] {\bullet} 
!{\save +<-13pt,0pt>*\txt{$\scriptstyle \{a,c\}$}  \restore}
[rr] {\bullet} 
!{\save +<13pt,0pt>*\txt{$\scriptstyle \{d,e\}$}  \restore}
[dd] {\bullet} 
!{\save +<13pt,0pt>*\txt{$\scriptstyle \{b,c\}$}  \restore}
[ll] {\bullet} 
!{\save +<-13pt,0pt>*\txt{$\scriptstyle \{d,f\}$}  \restore}
[rl] - [uu] - [rr] - [dd] - [ll] } $$

\medskip
\centerline{The simplicial complex $M(F)$}
\end{minipage} 
\end{center}
The complex $M(F)$ is homeomorphic to $S^1 \vee S^0$.
\end{example}

\subsection{Edge covering complex}

Let $G=(V,E)$ be a graph.  We define a simplicial complex $EC(G)$ on 
$E$ whose faces are the complements of the edge covers of $G$.  
For all $v \in V$, let ${\rm star} (v) = \bigl\{ e \in E \,|\, v \in e \bigr\}$; thus 
the minimal non-faces of 
$EC(G)$ are the minimal elements of $\bigl\{ {\rm star} (v) \,|\, v \in V \bigr\}$.  Note that 
if $v$ is an isolated vertex, then $EC(G) = \emptyset$.  
Let $e = \{x,y\} \in E$; then $\bigl( EC(G) : e \bigr) = EC(G \setminus e)$ since 
the minimal non-faces of $\bigl( EC(G) : e \bigr)$ are the minimal elements of 
$$ \bigl\{ {\rm star} (v) \,|\, v \in V , v \neq x,y \bigr\} \cup 
\bigl\{ {\rm star} (x) \setminus \{e\}, {\rm star} (y) \setminus \{e\} \bigr\} . $$
The complex $EC(G)$ is a cone with apex $e$ 
if and only if $x$ and $y$ are both adjacent to leaves.

\begin{theorem}
Let $F$ be a forest.  Then $EC(F)$ is a grape.  Moreover, $EC(F)$ is either 
contractible or homotopic to a sphere.
\end{theorem}

\begin{proof}
We may assume that $F$ has no isolated vertices, since $\emptyset $ is contractible.  
Proceed by induction on the number of edges of $F$.  If $F$ is a disjoint 
union of stars, then $EC(F) = \{\emptyset\}$, the $(-1)-$dimensional sphere.  
Otherwise, let $x_1 ,  \ldots , x_4$ be distinct vertices 
such that $\{x_1,x_2\}$, $\{x_2,x_3\}$, $\{x_3,x_4\}$ are edges and $x_1$ is 
a leaf.  If $x_4$ is a leaf, then $EC(F)$ is a cone with apex $\{x_2,x_3\}$ and we are 
done.  If $x_4$ is not a leaf, then $\{x_3,x_4\}$ dominates $\{x_2,x_3\}$ since 
$\bigl( EC(F) , \{x_3,x_4\} \bigr)$ is a cone with apex $\{x_2,x_3\}$.  Hence $EC(F)$ 
is homotopic to the suspension of $EC \bigl(F \setminus \{x_3,x_4\} \bigr)$, and we conclude 
by the inductive hypothesis.
\end{proof}

The following result relates the simplicial complex $EC(F)$ on $E$ to the simplicial 
complex $\ind (F)$ on $V$.  We let $\kappa (F)$ denote the number of connected 
components of $F$, or equivalently $\kappa (F) = |V| - |E|$.

\begin{theorem}
Let $F$ be a forest.  Then $EC(F)$ is homotopic to a sphere (resp.~contractible) 
if and only if $\ind (F)$ is homotopic to a sphere (resp.~contractible).  Moreover 
if $EC(F)$ is not contractible, the dimension of the sphere associated to $EC(F)$ 
is $i(F) - \kappa (F) - 1 = \gamma (F) - \kappa (F) - 1$.
\end{theorem}

\begin{proof}
We may assume that $F$ has no isolated vertices since in this case $EC(F) = \emptyset $, 
$\ind (F)$ is a cone and therefore they are both contractible.  
Proceed by induction on the number of edges of $F$.  If $F$ is a disjoint 
union of stars, then $EC(F) = \{\emptyset\}$, the $(-1)-$dimensional sphere, 
and $\ind (F) \simeq S^{\kappa (F) - 1}$ (see~\cite{MT1}).  Otherwise, let 
$x_1 ,  \ldots , x_4 \in V$ be such that $\{x_1,x_2\}$, $\{x_2,x_3\}$, $\{x_3,x_4\}$ 
are edges and $x_1$ is a leaf.  If $x_4$ is a leaf, then $EC(F)$ is a cone with 
apex $\{x_2,x_3\}$; $x_3$ dominates $x_4$ in $\ind (F)$ and both 
$(\ind (F) , x_3 \bigr)$ and $(\ind (F) : x_3 \bigr)$ are cones; thus $EC(F)$ and 
$\ind (F)$ are both contractible.  If $x_4$ is not a leaf, then $EC(F)$ is homotopic 
to $\Sigma \bigl( EC(F') \bigr)$, where $F' = F \setminus {\text{edge }} \{x_3,x_4\}$, 
while $\ind (F)$ is homotopic to $\ind (F')$ by Lemma~\ref{scremo}.  
By the inductive hypothesis we have that $EC \bigl( F' \bigr)$ and $\ind (F')$ are 
either both contractible or both homotopic to spheres and thus also 
$EC(F)$ and $\ind (F)$ have the same property.  Moreover if $EC(F)$ is not contractible, 
then it is homotopic to a sphere of dimension 
$\gamma (F') - \kappa (F') = \gamma (F) - \kappa (F) - 1$.  The equalities $i(F)=i(F')$
and $\gamma (F) = i(F)$, when $EC(F)$ and $\ind (F)$) are not contractible, 
follow from~\cite[Theorem~6.4]{MT1}.
\end{proof}

\subsection{Edge dominance complex}

Let $G=(V,E)$ be a graph.  We define a simplicial complex $ED(G)$ on $E$ 
whose faces are the complements of the dominating sets of the line dual of $G$.  
For all $e \in E$, let ${\rm star} (e) = \bigl\{ f \in E \,|\, f \cap e \neq \emptyset \bigl\}$; thus 
the minimal non-faces of $ED(G)$ are the minimal elements of 
$\bigl\{ {\rm star} (e) \,|\, e \in E \bigr\}$.

\begin{theorem}
Let $F$ be a forest.  Then $ED(F)$ is a grape.  Moreover $ED(F)$ is homotopic to a sphere 
of dimension $|E| - \beta_1 (F) - 1 = |E| - \alpha _0 (F) - 1$.
\end{theorem}

\begin{proof}
%\marginpar{numero di lati - matching number}
Proceed by induction on the number of edges of $F$.  If $F$ consists only of isolated vertices 
and edges, then $ED(F) = \{ \emptyset \}$, the $(-1)-$dimensional sphere, and the result is clear.  
Let $b$ be a leaf of $F$ and let $\{a,c\}$ be an edge of $F$ such that $a$ is adjacent 
to $b$ and $c \neq b$.  Since ${\rm star} \bigl( \{a,b\} \bigr) \subset {\rm star} \bigl( \{a,c\} \bigr)$, 
we deduce from Lemma~\ref{dodo} that $\{a,c\}$ dominates $\{a,b\}$.  The complex 
$\bigl( ED(F) , \{a,c\} \bigr)$ is a cone with apex $\{a,b\}$.  Since 
$\bigl( ED(F) : \{a,c\} \bigr) = ED \bigl( F \setminus {\text{edge }}\{a,c\} \bigr)$ and 
$ED(F) \simeq \Sigma \bigl( ED(F) : \{a,c\} \bigr)$, we conclude by induction that $ED(F)$ is a 
grape and that it is homotopic to a sphere.  To compute the dimension of the sphere, let 
$M \subset E$ be a matching of maximum cardinality and $b$ be a leaf adjacent to the vertex $a$.  
We may assume that the edge $\{a,b\}$ is not isolated.  If $\{a,b\} \in M$, then removing an 
edge $\{a,c\}$ with $c \neq b$ we may conclude by induction.  If $\{a,b\} \notin M$, then 
an edge $\{a,c\} \in M$ for exactly one $c$.  The set $M \cup \{a,b\} \setminus \{a,c\}$ is again 
a matching with same cardinality as $M$, and we may conclude as before.  The last equality follows 
by a similar argument or by Theorem~\ref{konig}.
\end{proof}

%{\bf Question.}  What is the graph theoretical invariant giving the dimension of the sphere 
%\marginpar{numero di lati - matching number}
%associated to $ED(F)$?

\subsection{Interval order complex}

Let $X$ be a finite set of closed bounded intervals in $\mathbb{R}$; the interval 
order complex on $X$ is the simplicial complex $\cO (X)$ whose faces are the 
subsets of $X$ consisting of pairwise disjoint intervals.  The simplicial complex 
$\cO(X)$ is shellable thanks to~\cite{BM}.  In particular, it follows that $\cO(X)$ 
is contractible or homotopic to a wedge of spheres.  We give a short 
direct computation of the homotopy type of $\cO (X)$.

Associated to $X$ there 
is also a graph $O(X) = (V,E)$, where $V=X$ and $\{ I,J \} \in E$ if and only if 
$I \cap J \neq \emptyset$.  Clearly, $\ind \bigl(O(X) \bigr) = \cO (X)$.  Theorem~\ref{inte} 
does not apply to $\ind \bigl( O(X) \bigr)$, since $O(X)$ is not in general a forest.  
Nevertheless we have the following result.

\begin{theorem}
The simplicial complex $\cO (X)$ is a grape.
\end{theorem}

\begin{proof}
If $X = \emptyset $, then the result is clear.  Otherwise 
let $I=[a,b] \in X$ be an interval such that $b = \min \bigl\{ y \;|\; [x,y] \in X \bigr\}$.  The 
vertices of $O(X)$ adjacent to $I$ are the intervals of $X$ containing $b$.  If 
no interval in $X \setminus \{ I \}$ contains $b$, then $\cO(X)$ is a 
cone with apex $I$ and we are done.  Otherwise, let $J \in X$ be an interval 
containing $b$.  By construction we have $N [I] \subset N [J]$ 
(in the graph $O(X)$) and by Lemma~\ref{indo} we deduce that $J$ dominates 
$I$ in $\cO(X)$.  Since $\bigl( \cO(X) , J \bigr) = \cO \bigl( X \setminus \{J\} \bigr)$ 
and $\bigl( \cO(X) : J \bigr) = \cO \bigl( X \setminus N [J] \bigr)$, we 
conclude by induction on the cardinality of $X$.
\end{proof}

\begin{example}
The simplicial complex $\cO (X)$ may be a wedge of spheres of different dimensions.  
Let $X = \bigl\{ [0,2] , [0,6] , [1,3] , [4,7] , [5,8] \bigr\}$.  The graph $O(X)$ 
and the simplicial complex $\cO (X)$ are depicted in the following figure.
\begin{center}
\begin{minipage}{150pt}
$$ \xygraph {[] !~:{@{.}} 
!{<0pt,0pt>;<20pt,0pt>:} 
{\bullet} 
!{\save +<-11pt,0pt>*\txt{$\scriptstyle [0,2]$}  \restore}
[dd] {\bullet} 
!{\save +<-11pt,0pt>*\txt{$\scriptstyle [1,3]$}  \restore}
[urr] {\bullet} 
!{\save +<0pt,10pt>*\txt{$\scriptstyle [0,6]$}  \restore}
[urr] {\bullet} 
!{\save +<11pt,0pt>*\txt{$\scriptstyle [4,7]$}  \restore}
[dd] {\bullet} 
!{\save +<11pt,0pt>*\txt{$\scriptstyle [5,8]$}  \restore}
[rl] - [uu] - [dll] - [drr] [llll] - [uu] - [drr] - [dll] } $$

\medskip
\centerline{The graph $O(X)$}
\end{minipage} 
\hspace{10pt}
\begin{minipage}{150pt}
$$ \xygraph {[] !~:{@{.}} 
!{<0pt,0pt>;<20pt,0pt>:} 
[d] {\bullet} 
!{\save +<-11pt,0pt>*\txt{$\scriptstyle [0,6]$}  \restore}
[urr] {\bullet} 
!{\save +<-11pt,0pt>*\txt{$\scriptstyle [1,3]$}  \restore}
[rr] {\bullet} 
!{\save +<11pt,0pt>*\txt{$\scriptstyle [4,7]$}  \restore}
[dd] {\bullet} 
!{\save +<11pt,0pt>*\txt{$\scriptstyle [0,2]$}  \restore}
[ll] {\bullet} 
!{\save +<-11pt,0pt>*\txt{$\scriptstyle [5,8]$}  \restore}
[rl] - [uu] - [rr] - [dd] - [ll] } $$

\medskip
\centerline{The simplicial complex $\cO (X)$}
\end{minipage} 
\end{center}
The complex $\cO(X)$ is homeomorphic to $S^1 \vee S^0$.
\end{example}

\bigskip

We summarize in the following table the results obtained in this section on the 
homotopy types of the simplicial complexes associated to a (possibly multidirected) 
forest $F=(V,E)$ and of the interval order complex.  
Wedge of spheres means that the spheres have in general different dimensions and 
the wedge could be empty (i.e.~the simplicial complex could be contractible).

\medskip

\begin{center}
\begin{tabular}{|l|l|}
\hline
Name & Homotopy type \\[3pt]
\hline
\hline
Oriented forests & Wedge of spheres \\[3pt]
\hline
Independence complex & Contractible or sphere of dimension \\
& $i(F) - 1 = \gamma (F) - 1$ \\[3pt]
\hline
Dominance complex & Sphere of dimension \\
& $\alpha _0 (F) - 1 = \beta _1 (F) - 1$ \\[3pt]
\hline
Matching complex & Wedge of spheres \\[3pt]
\hline
Edge covering complex & Contractible or sphere of dimension \\
& $|E| - |V| + i(F) - 1 = |E| - |V| + \gamma (F) - 1$ \\[3pt]
\hline
Edge dominance complex & Sphere of dimension \\
& $|E| - \alpha _0 (F) - 1 = |E| - \beta _1 (F) - 1$ \\[3pt]
\hline
Interval order complex & Wedge of spheres \\[3pt]
\hline 
\end{tabular}
\end{center}

\end{document}